\documentclass[12pt]{amsart}

\usepackage{amsmath}
\usepackage{amsfonts}
\usepackage{amssymb}
\usepackage{amsthm}
\usepackage{hyperref}
\usepackage[all]{hypcap}
\usepackage{youngtab}
\usepackage{enumerate}
\usepackage{cleveref}
\usepackage{hyperref}
\usepackage{mathrsfs}
\usepackage[top=0.9in, bottom=1.15in, left=1.1in, right=1.1in]{geometry}
\hypersetup{colorlinks=true,linkcolor=blue,citecolor=magenta}
\newtheorem{theorem}{Theorem}[section]

\newtheorem{corollary}[theorem]{Corollary}

\newtheorem*{remark}{Remark}
\newtheorem*{remarks}{Remarks}

\Crefname{conjecture}{Conjecture}{Conjectures}
\numberwithin{figure}{section}
\theoremstyle{plain}

\theoremstyle{plain}

\newcommand{\N}{\mathbb{N}}

\newcommand{\R}{\mathbb{R}}

\numberwithin{equation}{section}

\author{Larry Rolen}
\address{212 McAllister Building
The Pennsylvania State University
University Park, PA 16802} 
\email{larryrolen@psu.edu}
\title{On $t$-core towers and $t$-defects of partitions}

\begin{document}

\bibliographystyle{amsplain}
\maketitle
\begin{abstract}
We study generating functions which count the sizes of $t$-cores of partitions, and, more generally, the sizes of higher rows in $t$-core towers. We then use these results to derive an asymptotic for the average size of the $t$-defect of partitions, as well as some curious congruences.
\end{abstract}
\section{Introduction and Statement of results}

The study of generating functions formed by sums or products over sets of partitions has a long and beautiful history going back to Euler and encompassing many arithmetic, analytic, and modularity results (see \cite{AndrewsBook,BacherManivel,BlochOkounkov,Fine,Han,MacDonald,NekrasovOkounkov,Robert,Zagier} and the references therein for just a few examples related to the generating functions discussed here). Here we consider generating functions which enumerate statistics associated to the so-called $t$-core towers of partitions. To describe this, we denote by $\mathcal P$ the set of all integer partitions (including the empty partition $\emptyset$), and for a generic partition $\lambda$, we denote by $|\lambda|$ the size of $\lambda$, i.e., the sum of the parts of $\lambda$ (with $|\emptyset|=0$).
We will also write $\lambda\vdash n$ if $|\lambda|=n$. 
We then consider the generating functions

\[
T_{j,t}(q):=\sum_{\lambda\in\mathcal P}|\beta_j(t;\lambda)|q^{|\lambda|}
,
\]
where $\beta_j(t;\lambda)$ is the $j$-th row of the $t$-core tower of $\lambda$ (defined in Section \ref{CoreTowerPrelim}), and where for any set of partitions $X$, we define $|X|$ as the sum of sizes of all partitions in $X$. In particular, $T_{0,t}(q)$ is the generating function which weights partitions by the size of their $t$-cores. 

Our main result is then as follows, where the $q$-series $G_2^0$ is defined by 
\[
G_2^0(q)
:=
\sum_{n\geq1}\frac{nq^n}{1-q^n}
=
\sum_{n\geq1}\sigma_1(n)q^n
\]
with $\sigma_1(n):=\sum_{d|n}d$, 
and where
\[
(a)_n=(a;q)_n:=\prod_{j=0}^{n-1}\left(1-aq^j\right)
.
\]
\begin{theorem}
\label{mainthm}
For all $t\geq2$, $j\geq0$, we have that
\[
T_{j,t}(q)
=
\frac{t^{j}G_2^0\left(q^{t^j}\right)-t^{j+2}G_2^0\left(q^{t^{j+1}}\right)}{(q)_{\infty}}
.
\]
\end{theorem}
\begin{remark}
It is interesting to note that the $j=0$ case of Theorem \ref{mainthm} is related to Theorem 1.1 of \cite{BacherManivel}, and can indeed be deduced from that result after combining with Proposition 3.6 of \cite{Olsson}. 
\end{remark}

In particular, a number of authors have been interested in the positivity and asymtptoics of the number of $t$-cores of size $n$ (cf. e.g. \cite{Anderson,BerkovichGarvan,KenGranville, HanOno,HanusaNath,Kim,KimRouse,Ono}), which has important implications for the representation theory of the symmetric group. Specificallly, the $t$-cores of size $n$ correspond to $t$-defect-zero blocks of the corresponding irreducible representations of $S_n$ \cite{JamesKerber}. Given the importance of these results, which establish the existence and approximate number of defect-zero blocks for $S_n$, it is natural to study the general growth properties of the $t$-defects of partitions, which are denoted by $d_t(\lambda)$ and defined in \eqref{DefectDefn}.

We then show the following, where
\[
D_t(q):=\sum_{\lambda\in\mathcal P}d_t(\lambda)q^{|\lambda|}
.
\]
\begin{corollary}\label{DefectCorGenFn}
For all $t\geq2$, we have
\[
D_t(q)=\frac{\sum_{j\geq1}t^jG_2^0\left(q^{t^j}\right)}{(q)_{\infty}}
.
\]
\end{corollary}
\begin{remark}
It is curious to note that this generating function is reminiscent of the series in Fourier-Jacobi expansions of Siegel modular forms (cf. Chapter II.6 of \cite{EichlerZagier}), as was pointed out to the author by Kathrin Bringmann. 
\end{remark}
Using a well-known Tauberian theorem of Ingham (described in Section \ref{TauberianPrelim}), we also deduce the following.
\begin{corollary}\label{CorAsymp}
For any $t\geq2$, the average size of the $t$-defect of $\lambda$ among the partitions of $n$ is asymptotic to $n/(t-1)$.
\end{corollary}
\begin{remark}
We have here chosen to highlight this particular asymptotic result due to its representation theoretic and historic interest. However, many further asymptotic results can be obtained by studying the generating functions in Theorem \ref{mainthm}. For example, as noted in the remark following Theorem \ref{mainthm}, for $j=0$ there is a connection via Proposition 3.6 of \cite{Olsson} to hook lengths of partitions. Using this idea, and a similar idea involving inserting Dirichlet characters into the generating functions of Theorem 1.1 of \cite{BacherManivel}, amusing results can be proven, such as that the residue classes modulo $m$ of hook lengths are on average equidistributed for any $m$. We leave the details to the interested reader.
\end{remark}
We also consider what we shall call {\it generalized }$(j,t)${\it -cores} (cf. Section \ref{CoreTowerPrelim}). The following then occurs as an intermediate step in the proof of Theorem \ref{mainthm}, where $c_{j,t}(n)$ is the number of generalized $(j,t)$-core partitions of size $n$.
\begin{corollary}\label{GenTCorePartitionsCor}
For any $t\geq2$, $j\geq0$, we have
\[
\sum_{n\geq0}c_{j,t}(n)q^n
=
\frac{\left(q^{t^{j+1}};q^{t^{j+1}}\right)_{\infty}^{t^{j+1}}}{(q)_{\infty}}
.
\]
\end{corollary}
\begin{remark}
Since generalized $(0,t)$-cores are the same as $t$-cores, we note that Corollary \ref{GenTCorePartitionsCor} implies the well-known and heavily used generating function for the number of $t$-cores.
\end{remark}

Much study has been devoted to congruences of generating functions of $t$-cores and relations to congruences of the ordinary partition function, just a few of the many results can be found in \cite{Boylan,Chen,GarvanMore,GarvanKimStanton,KenGranville,HirschhornSellers,KolitschSellers,RaduSellers}.
Theorem \ref{mainthm} also implies several interesting congruences. 
In particular, highlighting the case of $j=0$, we denote 
\[
a_{t}(n):=\sum_{\lambda\vdash n}|\beta_0(t;\lambda)|
,
\]
which is simply the sum of the sizes of all $t$-cores of partitions of $n$. In particular, we have the following, where $p(n)$ denotes the number of partitions of $n$.
\begin{corollary}\label{CongCor}
For all $t,n$, we have
\[
a_{t}(n)\equiv np(n)\pmod{t^2}
.
\]
In particular, we have that
\[
a_{t}(tn)\equiv0\pmod{t^2}
.
\]
\end{corollary}
\begin{remarks}
\begin{enumerate}
\item
The congruence
\[a_{t}(tn)\equiv0\pmod t\]
is obvious, since the $t$-core of each partition of size $tn$ is clearly of size divisible by $t$. The congruence modulo $t^2$, however, only holds when one sums over all partitions of $tn$.
\item
Further congruences can be deduced by using known congruences for $p(n)$.
\item  Note that the numerator of $T_{0,t}$ is a $p$-adic modular form, and more particularly that $T_{0,t}$ is congruent modulo any power of a prime $p^a$ (with $p\geq5$) to a weight $3/2$, level $tp^a$ weakly holomorphic modular form (this follows directly from the discussion of the completed weight $2$ Eisenstein series $E_2^*$ in Section \ref{EisensteinPrelim}). Thus, many congruences can be expected to hold beyond those in Corollary \ref{CongCor}
\item James Sellers has pointed out numerous further congruences to the author. For example, it is easy to check that $T_{0,t}$ is the product of the series $\sum_{n\geq1}nc_{0,t}(n)q^n$ and a $q$-series supported on coefficients of order divisible by $p$ (see also the proof of Theorem \ref{mainthm}). Then for each prime $5\leq p\leq23$, there are congruences of the form $a_{p}(pn+a)\equiv0\pmod \ell$ for any prime $\ell|(p-1)/2$ for certain values of $a$, corresponding to the congruences shown in beautiful work of Garvan on congruences for $c_{0,t}(n)$ \cite{Garvan}. We leave the details to the interested reader, and simply point out that it would be interesting to find a combinatorial ``explanation'' for such congruences.
\end{enumerate}
\end{remarks}

Finally, we point out a simple recursion, which connects the $t$-cores with $t$-regular partitions. Here, we let $p_t(n)$ denotes the number of $t$-regular partitions of $n$, which are the partitions of $n$ with no part divisible by $t$.
\begin{corollary}\label{RegularPartCor}
Assuming the notation above, we have
\[
a_t(n)
=
np(n)-t\sum_{\substack{0\leq j\leq n\\ t|j}}jp\left(\frac jt\right)p_t(n-j),
\]
\end{corollary}
\begin{remark}
It would be interesting to find a combinatorial bijection explaining this identity.

\end{remark}

The paper is organized as follows. In Section \ref{Preliminaries}, we recall the definitions of the $t$-core towers and $t$-defects of partitions, review Ingham's Tauberian Theorem, and give a few useful results on Eisenstein series and their asymptotic expansions. We conclude with proofs of the above results in Section \ref{Proofs}.

\section*{Acknowledgements} 
The author is indebted to James Sellers for his enlightening series of talks on $(s,t)$-cores, out of which the idea for this work arose, as well as for interesting discussions on this work and on congruences for the functions described here in particular. The author is also grateful to Kathrin Bringmann for helpful discussions and comments.

\section{Preliminaries}\label{Preliminaries}
\subsection{$t$-core Towers and $t$-defects}\label{CoreTowerPrelim}

We now define the objects in the statements of Theorem \ref{mainthm} and its corollaries. Wonderful expositions on the combinatorics described here as well as comprehensive examinations of the connections to the representation theory of symmetric groups may be found in \cite{JamesKerber,Olsson}. For a general partition $\lambda=(\lambda_1,\ldots,\lambda_k)$, written with the convention that $\lambda_1\geq\lambda_2\geq\ldots\geq\lambda_k$, we recall that there is an associated \emph{Young diagram} which consists of a grid of {\it cells} with $k$ rows, the $j$-th row of which consists of $\lambda_j$ boxes. Furthermore, for any cell $\mathfrak c$ in $\mathcal Y_{\lambda}$, the {\it hook length} is the number of cells to the right of $\mathfrak c$ plus the number of cells below $\mathfrak c$ plus one. For example, if $\lambda=(5,4,2,2,1)$, then the Young diagram $\mathcal Y_{\lambda}$ is displayed in Figure \ref{YoungDiagramEx}.

\begin{figure}[!htb]\label{YoungDiagramEx}
\caption{The Young diagram $\mathcal Y_{(5,4,2,2,1)}$ with hook lengths shown.}
\[
\young(97431,7521,42,31,1)
\]
\end{figure}

These hook lengths play an important role in representation theory as well as combinatorics. For example, it is well-known that partitions of $n$ are in one-to-one correspondence with complex irreducible representations of the symmetric group $S_n$. The famous Frame-Robinson-Thrall formula then implies that the corresponding dimensions of these representations are expressed in terms of the product of hook lengths of all cells in the corresponding Young diagrams.

For a natural number $t\geq2$, a partition $\lambda$ is  
called a $t$-\emph{core} if no hook lengths in the Young diagram of $\lambda$ are divisible by $t$. For a general partition $\lambda$, and for a fixed $t$, there is an associated $t$-core subpartition which is formed from $\lambda$ by successively deleting hooks of length divisible by $t$ (possibly yielding the empty partition). This is called the $t$-\emph{core} of $\lambda$ and we denote it by $\lambda_{(t)}$. There is also an associated $t${\it -quotient}, denoted by $\lambda^{(t)}$, which is a $t$-tuple of partitions. The uniqueness of $\lambda_{(t)}$ follows from an elegant argument involving $t$-abaci. For this argument, and for more on the connections between partitions, $t$-cores, and representation theory, as well as the definition of $\lambda^{(t)}$, the reader is referred to \cite{JamesKerber}. 

We shall frequently have use of the the fact that partitions are uniquely built up from their $t$-cores and $t$-quotients.
\begin{theorem}[cf. \cite{JamesKerber}, Theorem 2.7.30]\label{TQuotientTheorem}
A partition $\lambda$ is uniquely determined by its $t$-core $\lambda_{(t)}$ and $t$-quotient $\lambda^{(t)}$. Conversely, given a $t$-core $\lambda$ and a $t$-tuple of partitions $(\lambda_1,\ldots,\lambda_t)$, there is a partition $\mu$ with 
\[
\mu_{(t)}=\lambda
\]
and
\[
\mu^{(t)}=(\lambda_1,\ldots,\lambda_t)
.
\]
Moreover, we have
\[
|\lambda|=|\lambda_{(t)}|+t|\lambda^{(t)}|
.
\]
\end{theorem}

Here, we follow Olsson \cite{Olsson} to further describe the $t$-core towers of partitions. We first introduce an object which we will find convenient, and which we shall call the $t${\it -core pre-tower} of $\lambda$. To define this, we begin by setting the $0$-th row of the $t$-core pre-tower as $\lambda$, and denote this by $\alpha_0(t;\lambda)$ (and by $\alpha_0$ when $t$ and $\lambda$ are clear from context). We then define for any $j$ the $j$-th row $\alpha_j$ of the $t$-core pre-tower by taking in order the $t$-quotients of all partitions in $\alpha_{j-1}$. It is well-known that this process must terminate, i.e., that $\alpha_j=0$ for some $j$. We then define the $t$-core pre-tower as the tower of all rows $\alpha_j$ with $|\alpha_j|\neq0$. Similarly, we define the $t${\it -core tower} of $\lambda$ as the tower whose $j$-th row $\beta_j=\beta_j(t;\lambda)$ is formed by taking the $t$-cores of the partitions in $\alpha_j(t;\lambda)$.
This procedure is illustrated by Figures \ref{PreCoreTowerEx} and \ref{CoreTowerEx}.

\begin{figure}[!htb]\label{PreCoreTowerEx}
\caption{The $2$-core pre-tower of $(5,4,2,2,1)$}
\begin{center}
\begin{tabular}{ c c c c c c c  }
& & & $(5,4,2,2,1)$ & & & \\
& $(1,1)$ & & & &  $(2)$&  \\
$(1)$ &&  $\emptyset$ &&   $\emptyset$&&  $(1)$  \\
\end{tabular}
\end{center}
\end{figure}

\begin{figure}[!htb]
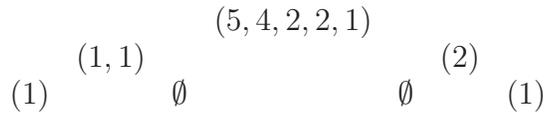
\label{CoreTowerEx}
\caption{The $2$-core tower of $(5,4,2,2,1)$}
\begin{center}
\begin{tabular}{ c c c c c c c  }
&&& $(3,2,1)$ &&& \\ 
& $\emptyset$ &&&& $\emptyset$ & \\
$(1)$ && $\emptyset$ && $\emptyset$ && (1)
\end{tabular}
\end{center}
\end{figure}

An important invariant of the representation of $S_n$ corresponding to a partition $\lambda$ is the so-called $t${\it -core defect}, which by (4) on page 43 of \cite{Olsson} is given for prime $t$ by

\begin{equation}\label{DefectDefn}
d_t(\lambda)=\frac{|\lambda|-\sum_{j\geq0}|\beta_j(t;\lambda)|}{t-1}
.
\end{equation}
We will analogously define $d_t(\lambda)$ by this same formula for general $t\geq2$, as it is still of combinatorial interest, if not of clear representation theoretic importance.
Finally, we will say that a partition $\lambda$ is a {\it generalized }$(j,t)${\it -core} if $\alpha_{j+1}=0$ (this is just a convenient way of saying that $|\beta_{k}|=0$ for all $k\geq j+1$). Note that generalized $(0,t)$-cores are the same as $t$-cores. 

\subsection{Ingham's Tabuerian Theorem}\label{TauberianPrelim}
We now recall the main result on asymptotics of $q$-series which we need to prove Corollary \ref{CorAsymp}. This is known as \emph{Ingham's Tauberian Theorem}, and it allows one to determine explicit asymptotics for the coefficients of a $q$-series with monotonic coefficients whose asymptotic expansion near $q=1$ is known. 
\begin{theorem}[Ingham, special case of Theorem 1 of \cite{Ingham}]\label{InghamTheorem} 
Let $f(q)=\sum_{n\geq0}a_nq^n$ be a power series with non-negative, monotonically increasing coefficients and radius of convergence $1$. Then if there are constants $A>0$, $\ell,\alpha\in\R$ with
\[
f\left(e^{-\varepsilon}\right)\sim \ell\varepsilon^{\alpha}e^{\frac{A}{\varepsilon}}\quad\mathrm{ as }\ \varepsilon\rightarrow0^+
,
\]
then for $n\rightarrow\infty$ we have
\[
a_n\sim\frac{\ell A^{\frac{\alpha}2+\frac14}}{2\sqrt\pi n^{\frac{\alpha}2+\frac34}}e^{2\sqrt{An}}
.
\]
\end{theorem}
In particular, this directly yields the celebrated Hardy-Ramanujan asymptotic for the partition function. To see this, recall that 
\begin{equation}\label{PartitionGenFn}
\frac1{(q)_{\infty}}=\sum_{n\geq0}p(n)q^n
,
\end{equation}
and that $\eta(q):=q^{\frac{1}{24}}(q)_{\infty}$ is a {\it modular form} of weight $1/2$. From the transformation of $\eta$ under modular inversion, it is easy to check that
\begin{equation}\label{PartitionGenFunctionAsymp}
\frac1{\left(e^{-\varepsilon};e^{-\varepsilon}\right)_{\infty}}
\sim
\frac{\varepsilon^{\frac12}e^{\frac{\pi^2}{6\varepsilon}}}{\sqrt{2\pi}}
,
\end{equation}
from which the Hardy-Ramanujan asymptotic 
\[
p(n)
\sim
\frac1{4n\sqrt 3}e^{\pi\sqrt{\frac{2n}3}}
\]
follows by directly plugging in to Theorem \ref{InghamTheorem} (together with the postitivity and monotonicity of $p(n)$, which are obvious from its combinatorial interpretation).
\subsection{Eisenstein series and asymptotics}\label{EisensteinPrelim}
In order to apply Theorem \ref{InghamTheorem}, we shall use \eqref{PartitionGenFunctionAsymp} together with modularity properties of the series $G_2^0$, which we now review. In fact, the notation $G_2^0$ is chosen to reflect that it is effectively the weight two Eisenstein series, which is denoted $G_2$, without its constant term. We shall rewrite this series in terms of the normalized weight $2$ Eisenstein series, defined by 
\[
E_2(\tau)
:=
1-24\sum_{n\geq1}\sigma_1(n)q^n
,
\]
where $q:=e^{2\pi i \tau}$. We then directly find that
\[
G_2^0(q)=\frac1{24}-\frac{E_2(\tau)}{24}
.
\]
Now, it is well-known that $E_2$ is not a modular form, but that is a nearly-modular object known as a {\it quasimodular form}. In particular, if we denote the imaginary part of $\tau$ by $y$, then we define a simple non-holomorphic modification of $E_2$ by letting
\[
E_2^*(\tau):=E_2(\tau)-\frac3{\pi y}
.
\]
This ``correction'' is then modular of weight $2$, and in particular satisfies 
\begin{equation*}
E_2^*\left(\frac{-1}{\tau}\right)=\tau^2E_2^*(\tau)
.
\end{equation*}
Using this, we directly find the following, where $\mathcal E_m(\tau):=E_2^*(m\tau)$:
\begin{equation}\label{E2StarTrans}
\mathcal E_m\left(\frac{-1}{\tau}\right)=\frac{\tau^2}{m^2}E_2^*\left(\frac{\tau}{m}\right)
.
\end{equation}
Letting $\tau=\frac{2\pi i }{\varepsilon}$, we find that $e(-1/\tau)=e^{-\varepsilon}$ (where $e(x):=e^{2\pi i x}$), so that by \eqref{E2StarTrans}, and where for a general function $f$ of $\tau$ we denote by $f|_{q=e(x)}$ the value $f(x)$, we find that
\[
\mathcal E_m\big|_{q=e^{-\varepsilon}}=-\frac{4\pi^2}{m^2\varepsilon^2}E_2^*\big|_{q=e^{-\frac{4\pi^2}{\varepsilon m}}}
.
\]
Hence for any $m\in\N$ we have
\begin{equation}\label{G2Trans}
\begin{split}
G_2^0(q^m)\big|_{q=e^{-\varepsilon}}
&
=
\left(\frac1{24}-\frac{\mathcal E_m(\tau)}{24}+\frac{1}{8\pi y m}\right)\Big|_{q=e^{-\varepsilon}}
\\
&
=
\frac1{24}+\frac{\pi^2}{6m^2\varepsilon^2}\left(1-24\sum_{n\geq1}\sigma_1(n)e^{\frac{-4\pi^2 n}{\varepsilon m}}\right)+\frac{1}{4m\varepsilon}
\end{split}
\end{equation}
After proving Theorem \ref{mainthm} and Corollary \ref{DefectCorGenFn}, we will use \eqref{PartitionGenFunctionAsymp} and \eqref{G2Trans} to prove Corollary \ref{CorAsymp}.

\section{Proof of Theorem \ref{mainthm} and its corollaries}\label{Proofs}
We are now ready to prove our main result.
\begin{proof}[Proof of Theorem \ref{mainthm}]
The proof follows from ideas which have their origins in the classical theory of $t$-cores, (see e.g. pg. 13 of \cite{MacDonald})). We first note that by repeated application of 
Theorem \ref{TQuotientTheorem}, for any $j\geq0$, $t\geq2$, we have
\begin{equation}
\label{GenTQuotFormula}
|\lambda|=|\beta_0|+t|\beta_1|+\ldots+t^j|\beta_j|+t^{j+1}|\alpha_{j+1}|
.
\end{equation}
Moreover, suppose we are given sets of partitions $B_0,\ldots,B_j$ with $t^k$ partitions in each $B_k$. Then given another set of $t^{j+1}$ partitions, say $A_{j+1}$, there is a unique partition $\lambda$ with $\beta_j=B_j$ for $j=0,\ldots,k$ and with $\alpha_{j+1}=A_{j+1}$. This partition can be determined by ``snaking'' through the $t$-core pre-tower and the $t$-core tower. That is, by Theorem \ref{TQuotientTheorem}, the rows $\alpha_{j+1}$ and $\beta_j$ uniquely determine $\alpha_j$, which together with $\beta_{j-1}$ determines $\alpha_{j-1}$. Continuing this process until we have uniquely determined $\alpha_0$ gives us the desired $\lambda$. Moreover, by \eqref{GenTQuotFormula}, we see that 
\[
\lambda
\vdash 
\left(|B_0|+t|B_1|+\ldots+t^j|B_j|+t^{j+1}|A_{j+1}|\right)
.
\]
Now for a general $m\in\N$, the generating function 
 \[
 \frac1{(q)_{\infty}^m}=:\sum_{n\geq0}b_{m}(n)q^n
 \]
 counts the number of partitions into $m$ colors, or equivalently, $b_{m}(n)$ is the number of $m$-tuples of partitions with total size $n$. We then say for two partitions $\lambda$ and $\mu$ that 
 \[
 \lambda\sim_j\mu
 \]
 if and only if
 \[
 \beta_k(t;\lambda)
 =
 \beta_k(t;\mu)
 \]
 for $k=0,1,\ldots,j$. 
 Hence, it is apparent by the discussion following \eqref{GenTQuotFormula} that for any $\lambda\in\mathcal P$, we have
\begin{equation}\label{FirstEqnProof11}
\sum
_
{
\substack{
\mu\in\mathcal P\\ \mu\sim_j\lambda
}
}
q^{|\mu|}
=
\frac{\prod_{k=0}^jq^{t^k|\beta_k(t;\lambda)|}}
{\left(q^{t^{j+1}};q^{t^{j+1}}\right)_{\infty}^{t^j+1}}
.
\end{equation}
Multiplying \eqref{FirstEqnProof11} by $\left(q^{t^{j+1}};q^{t^{j+1}}\right)_{\infty}^{t^j+1}$ and summing over all generalized $(j,t)$-cores, we find that 
\begin{equation}\label{gencoregenfunction}
\sum_{\substack{\lambda\in\mathcal P\\ |\alpha_{j+1}(t;\lambda)|=0}}\prod_{k=0}^jq^{t^k|\beta_k(t;\lambda)|}
=
\sum_{n\geq0}c_{j,t}(n)q^n
=
\frac{\left(q^{t^{j+1}};q^{t^{j+1}}\right)_{\infty}^{t^j+1}}
{(q)_{\infty}}
.
\end{equation}
Here, we have used \eqref{GenTQuotFormula} to deduce that the size of a generalized $(j,t)$-core is $\sum_{k=0}^jt^k|\beta_k|$, together with the obvious fact that each partition is $\sim_j$ to a unique generalized $(j,t)$-core.
Hence, once again using \eqref{GenTQuotFormula} and the discussion following that equation, together with \eqref{gencoregenfunction}, we find the following weighted version of \eqref{gencoregenfunction}: 
\begin{equation}\label{PenultEqn}
\begin{aligned}
\sum_{\lambda\in\mathcal P}\sum_{k=0}^jt^k|\beta_k(t;\lambda)|q^{|\lambda|}
&
=
\frac{\displaystyle\sum_{\substack{\lambda\in\mathcal P\\ |\alpha_{j+1}(t;\lambda)|=0}}q\frac d{dq}\left(\prod_{k=0}^jq^{t^k|\beta_k|}\right)}{\left(q^{t^{j+1}};q^{t^{j+1}}\right)_{\infty}^{t^{j+1}}}
\\
&
=
\left(\frac{1}{\left(q^{t^{j+1}};q^{t^{j+1}}\right)_{\infty}^{t^j+1}}\right)
q\frac{d}{dq}
\left(\frac{\left(q^{t^{j+1}};q^{t^{j+1}}\right)_{\infty}^{t^j+1}}
{(q)_{\infty}}\right)
.
\end{aligned}
\end{equation}
Now it is a straightforward calculation (and a well-known fact) that the logarithmic derivative of $(q)_{\infty}$ is
\begin{equation}\label{G2OverEta}
\frac{q\frac d{dq}\left((q)_{\infty}\right)}{(q)_{\infty}}
=-G_2^0(q)
.
\end{equation}
Combining with \eqref{PenultEqn}, and plugging in $j=0$, we have shown that 
\[
T_{0,t}(q)=\frac{G_2^0(q)-t^2G_2^0(q^t)}{(q)_{\infty}}
.
\]
We now proceed by induction on $j$.
Assume that Theorem \ref{mainthm} is true for $j-1$. Then we find by \eqref{PenultEqn} and \eqref{G2OverEta} that
\begin{equation}\label{FirstNeed}
\frac{G_2^0(q)-t^{2j+2}G_2^0\left(q^{t^{j+1}}\right)}{(q)_{\infty}}
=
\sum_{k=0}^j\sum_{\lambda\in\mathcal P}t^k|\beta_k(t;\lambda)|q^{|\lambda|}=\sum_{k=0}^jt^kT_{k,t}(q)
,
\end{equation}
which by the induction hypothesis is equal to 
\[
t^jT_{j,t}(q)+\frac{1}{(q)_{\infty}}\sum_{k=0}^{j-1}\left(t^{2k}G_2^0\left(q^{t^k}\right)-t^{2k+2}G_2^0\left(q^{t^{k+1}}\right)\right)
,
\]
which telescopes to 
\begin{equation}\label{SecondNeed}
t^jT_{j,t}(q)+\frac{1}{(q)_{\infty}}\left(G_2^0(q)-t^{2j}G_2^0\left(q^{t^{j}}\right)\right)
.
\end{equation}
Combining \eqref{FirstNeed} and \eqref{SecondNeed} directly yields the claim of Theorem \ref{mainthm}.
\end{proof}

The proof of Corollary \ref{DefectCorGenFn} is now immediate.
\begin{proof}[Proof of Corollary \ref{DefectCorGenFn}]
First note that \eqref{PartitionGenFn} and \eqref{G2OverEta} together imply that
\begin{equation}\label{GenFunctionG2EtanPn}
\frac{G_2^0(q)}{(q)_{\infty}}=\sum_{n\geq1}np(n)q^n
.
\end{equation}

By Theorem \ref{mainthm} and \eqref{DefectDefn}, we find that 
\begin{equation*}
\begin{aligned}
D_t(q)
&
=
\sum_{\lambda\in\mathcal P}q^{|\lambda|}\left(\frac{|\lambda|-\sum_{j\geq0}|\beta_j|}{t-1}\right)
\\
&
=
\frac{1}{(t-1)(q)_{\infty}}
\left[
G_2^0(q)
+
\sum_{j\geq0}
\left(
t^{j+2}G_2^0\left(q^{t^{j+1}}\right)
-
t^jG_2^0\left(q^{t^j}\right)
\right)
\right]
,
\end{aligned}
\end{equation*}
which telescopes to form the claimed series in Corollary \ref{DefectCorGenFn}.

\end{proof}

We now apply the discussion of Sections \ref{TauberianPrelim} and \ref{EisensteinPrelim} to Corollary \ref{DefectCorGenFn} in order to prove Corollary \ref{CorAsymp}.
\begin{proof}[Proof of Corollary \ref{CorAsymp}]
It is clear that $D_t(q)$ is convergent for $|q|<1$. We now establish that the coefficients of $D_t(n)$ are monotonically increasing, noting in particular that their positivity is clear from their combinatorial interpretation. To see this, note that it is enough to show that for any $m$, the coefficients of 
\[
F_m(q):=\frac{G_2^0(q^m)}{(q)_{\infty}}
\]
are monotonic and positive. We show this by using \eqref{GenFunctionG2EtanPn} that
\begin{equation}
\label{FMConv}
F_m(q)=\left(\sum_{n\geq1}np(n)q^{mn}\right)\left(\frac{(q^m;q^m)_{\infty}}{(q)_{\infty}}\right)
.
\end{equation}
Now if we have two sequences $\{a_n\}_{n\geq0}$ and $\{b_n\}_{n\geq0}$ which are both nonnegative with the coefficients $b_n$ monotonically increasing, then the coefficients $c_n$ of the Cauchy products are also positive and monotonically increasing, as 
\[
c_{n+1}=a_0b_{n+1}+a_1b_n+\ldots+a_{n+1}b_0>a_0b_n+a_1b_{n-1}+\ldots+a_nb_0=c_n
.
\]
Now the sequence $\{np(n)\}_{n\geq0}$ is obviously non-negative and monotonically increasing, and the coefficients of $(q^m;q^m)_{\infty}/(q)_{\infty}$ are non-negative since they count $m$-regular partitions, which follows from the well-known identity:
\begin{equation}\label{RegPartsGenFunction}
\frac{(q^m;q^m)_{\infty}}{(q)_{\infty}}=\sum_{n\geq0}p_t(n)q^n
.
\end{equation}

Thus, we see that the coefficients of $D_t(q)$ are positive and monotonically increasing, and hence the conditions of Theorem \ref{InghamTheorem} apply. We now need to determine the asymptotic main term of $D_t(e^{-\varepsilon})$.
Using  \eqref{G2Trans}, we find that
\begin{equation*}
\begin{aligned}
\left[\sum_{j\geq1}t^jG_2^0\left(q^{t^j}\right)\right]\Bigg|_{q=e^{-\varepsilon}}
&
=
\sum_{j\geq1}t^j
\left(
\frac1{24}+\frac{\pi^2}{6\varepsilon^2t^{2j}}\left(1-24\sum_{n\geq1}\sigma_1(n)e^{\frac{-4\pi^2 n}{t^j \varepsilon}}\right)+\frac{1}{4t^j\varepsilon}
\right)
,
\end{aligned}
\end{equation*}
so as $\varepsilon\rightarrow0^+$, we easily see that
\begin{equation}\label{OverallAsymp}
\left[\sum_{j\geq1}t^jG_2^0\left(q^{t^j}\right)\right]\Bigg|_{q=e^{-\varepsilon}}
\sim\frac{\pi^2}{6\varepsilon^2}
\sum_{j\geq1}t^{-j}
=
\frac{\pi^2}{6(t-1)\varepsilon^2}
.
\end{equation}

By combining \eqref{OverallAsymp} with \eqref{PartitionGenFunctionAsymp} and Corollary \ref{DefectCorGenFn},
we find that
\[
D_t(e^{-\varepsilon})
\sim
\frac{\pi^{\frac32}}{6\sqrt{2}(t-1)\varepsilon^{\frac32}}e^{\frac{\pi^2}{6\varepsilon}}
.
\]
Hence, by Theorem \ref{InghamTheorem}, if we set 
\[
d_t(n):=\sum_{\lambda\vdash n}d_t(\lambda)
,
\] 
we find that 
\[
d_t(n)
\sim
\frac{\sqrt{3}}{12(t - 1)}e^{\pi\sqrt{\frac{2n}3}}
,
\]
and in particular that 
\[
d_t(n)\sim
\frac{np(n)}
{t-1}
,
\]
from which the claimed result on the average values of $d_t(\lambda)$ for partitions $\lambda\vdash n$ follows.

\end{proof}

We have also essentially already shown Corollary \ref{GenTCorePartitionsCor}.
\begin{proof}[Proof of Corollary \ref{GenTCorePartitionsCor}]
The claim of Corollary \ref{GenTCorePartitionsCor} is contained in \eqref{gencoregenfunction}.
\end{proof}

The proof of Corollary \ref{CongCor} is now also straightforward.
\begin{proof}[Proof of Corollary \ref{CongCor}]
The congruence follows by plugging in $j=0$ into Theorem \ref{mainthm} and using \eqref{GenFunctionG2EtanPn}.
\end{proof}

Finally, Corollary \ref{RegularPartCor} also follows directly.
\begin{proof}[Proof of Corollary \ref{RegularPartCor}]
To see this, consider the statement of Theorem \ref{mainthm} for $j=0$ and apply \eqref{GenFunctionG2EtanPn} and  \eqref{FMConv}.
\end{proof}


\end{document}